\documentclass[a4paper, reqno]{amsart}

\usepackage[british]{babel}
\usepackage{amsfonts}
\usepackage{amsmath}
\usepackage{amsthm}
\usepackage{amssymb}
\usepackage[all]{xy}
\usepackage{calrsfs}
\usepackage[mathscr]{euscript}

\theoremstyle{plain}
\newtheorem{theorem}[subsection]{Theorem}
\newtheorem{proposition}[subsection]{Proposition}

\theoremstyle{definition}
\newtheorem{remark}[subsection]{Remark}
\newtheorem{construction}[subsection]{Construction}

\newcommand{\comp}{\raisebox{0.2mm}{\ensuremath{\scriptstyle{\circ}}}}
\newcommand{\defn}{\textbf}
\newcommand{\To}{\Rightarrow}
\newcommand{\del}{\ensuremath{\partial}}
\renewcommand{\ker}{\ensuremath{\mathsf{Ker\,}}}
\newcommand{\Fun}{\ensuremath{\mathsf{Fun}}}
\newcommand{\Arrm}{\ensuremath{\mathsf{Arr}^{m}\!}}
\newcommand{\Extm}{\ensuremath{\mathsf{Ext}^{m}\!}}
\newcommand{\CExtm}{\ensuremath{\mathsf{CExt}^{m}_{\B}}}
\newcommand{\A}{\ensuremath{\mathcal{A}}}

\newcommand{\nlie}{\ensuremath{{}_{n}\mathsf{lie}}}
\newcommand{\lie}{\ensuremath{\mathsf{lie}}}
\newcommand{\Ch}{\ensuremath{\mathsf{Ch}}}
\renewcommand{\S}{\ensuremath{\mathsf{S}}}
\newcommand{\Set}{\ensuremath{\mathsf{Set}}}
\newcommand{\G}{\ensuremath{\mathbb{G}}}
\newcommand{\K}{\ensuremath{\mathbb{K}}}
\newcommand{\Hbb}{\ensuremath{\mathbb{H}}}
\newcommand{\B}{\ensuremath{\mathcal{B}}}
\newcommand{\D}{\ensuremath{\mathsf{d}}}
\newcommand{\Lie}{\ensuremath{\mathsf{Lie}}}
\newcommand{\nLie}{\ensuremath{{}_{n}\mathsf{Lie}}}

\newcommand{\BB}{\ensuremath{\mathscr{B}}}
\renewcommand{\AA}{\ensuremath{\mathscr{A}}}
\newcommand{\LL}{\ensuremath{\mathscr{L}}}
\newcommand{\HH}{\ensuremath{\mathscr{H}}}
\newcommand{\NN}{\ensuremath{\mathscr{N}}}
\newcommand{\KK}{\ensuremath{\mathscr{K}}}
\newcommand{\UU}{\ensuremath{\mathscr{U}}}
\newcommand{\RR}{\ensuremath{\mathscr{R}}}
\newcommand{\FF}{\ensuremath{\mathscr{F}}}
\newcommand{\Leib}{\ensuremath{\mathsf{Lb}}}
\newcommand{\nLeib}{\ensuremath{{}_{n}\mathsf{Lb}}}
\newcommand{\abnleib}{\ensuremath{\mathsf{ab}_{n}\mathsf{lb}}}
\newcommand{\abnlie}{\ensuremath{\mathsf{ab}_{n}\mathsf{lie}}}
\newcommand{\Vect}{\ensuremath{\mathsf{Vect}}}
\newcommand{\nHL}{\ensuremath{{}_{n}\mathit{HL}}}
\newcommand{\HL}{\ensuremath{\mathit{HL}}}

\newcommand{\sgn}{\ensuremath{\mathrm{sgn}}}
\newcommand{\CExt}{\ensuremath{\mathsf{CExt}}}
\newcommand{\Ext}{\ensuremath{\mathsf{Ext}}}
\newcommand{\Arr}{\ensuremath{\mathsf{Arr}}}
\newcommand{\tensor}{\otimes}
\newcommand{\noproof}{\qed}

\newdir{>}{{}*:(1,-.2)@^{>}*:(1,+.2)@_{>}}
\newdir{<}{{}*:(1,+.2)@^{<}*:(1,-.2)@_{<}}

\begin{document}

\title[Some results on homology of Leibniz and Lie $n$-algebras]{Some results on homology of\\ Leibniz and Lie $n$-algebras}

\author[J.~M.~Casas]{Jos\'e~Manuel Casas}
\email{jmcasas@uvigo.es}
\address{Departamento de Matem\'atica Aplicada I, Universidad de Vigo, E.U.I.T. Forestal\\
Campus Universitario A Xunqueira, 36005 Pontevedra, Spain}
\author[E.~Khmaladze]{Emzar Khmaladze}
\email{khmal@rmi.acnet.ge}
\address{Department of Algebra, A.~Razmadze Mathematical Institute\\
M.~Alexidze St.~1, 0193 Tbilisi, Georgia}
\address{Departamento de Matem\'atica Aplicada I, Universidad de Vigo, E.U.I.T. Forestal\\
Campus Universitario A Xunqueira, 36005 Pontevedra, Spain}
\author[M.~Ladra]{Manuel Ladra}
\email{manuel.ladra@usc.es}
\address{Departamento de \'Algebra, Universidad de Santiago de Compostela\\
15782 Santiago de Compostela, Spain}
\author[T.~Van~der Linden]{Tim Van~der Linden}
\email{tvdlinde@vub.ac.be}
\address{Centro de Matem\'atica, Universidade de Coimbra\\
3001--454 Coimbra, Portugal}

\thanks{The first three authors' research was supported by Ministerio de Educacion y Ciencia under grant number
MTM2009-14464-C02-02 (includes European FEDER support), by Xunta de Galicia under grant number Incite09 207 215 PR and
by project Ingenio Mathematica (i-MATH) under grant number CSD2006-00032 (Consolider Ingenio 2010).
The fourth author's research was supported by Centro de Matem\'atica da Universidade de Coimbra and by Funda\c c\~ao para a Ci\^encia e a Tecnologia (under grant number SFRH/BPD/38797/2007).}



\begin{abstract}
From the viewpoint of semi-abelian homology, some recent results on homology of Leibniz $n$-algebras can be explained categorically. In parallel with these results, we develop an analogous theory for Lie $n$-algebras. We also consider the relative case: homology of Leibniz $n$-algebras relative to the subvariety of Lie $n$-algebras.
\end{abstract}

\maketitle

\section{Introduction}
In his article~\cite{Casas:hwtcolna}, Casas studied a homology theory for Leibniz $n$-algebras~\cite{CLP} based on Leibniz homology~\cite{Loday-Leibniz, LP}. The definition of homology used there exploits the remarkable properties of the so-called \emph{Daletskii functor}
\[
\D_{n-1}\colon{\nLeib\to {}_{2}\Leib=\Leib}
\]
from the category of Leibniz $n$-algebras to the category of Leibniz $2$-algebras, i.e., ordinary Leibniz algebras. Among other results, the author obtained a Hopf formula expressing the second homology of a Leibniz $n$-algebra as a quotient of commutators. Later, using \v Cech derived functors~\cite{Donadze-Inassaridze-Porter}, Casas, Khmaladze and Ladra characterised the higher homology vector spaces in terms of higher Hopf formulae~\cite{CKL}.

Even though Lie $n$-algebras are very close to Leibniz $n$-algebras, it is not easy to extend these results from the Leibniz case to the Lie case, because presently a suitable functor---i.e., a functor with properties similar to $\D_{n-1}$---from the category $\nLie$ of Lie $n$-algebras to the category ${}_{2}\Lie=\Lie$ of Lie algebras is missing.

In this article we study the problem from a different point of view which does allow us to easily carry over results from the Leibniz case to the Lie case. First we reformulate the known results on homology of Leibniz $n$-algebras in terms of categorical Galois theory, and then we can use the same proofs to compute the homology of Lie $n$-algebras.

It was shown in the article~\cite{CKL} that the homology of Leibniz $n$-algebras from~\cite{Casas:hwtcolna} coincides with their Quillen homology. Since this Quillen homology is equivalent to comonadic homology relative to the comonad induced by the forgetful/free adjunction to $\Set$, and since the categories of Leibniz and Lie $n$-algebras are semi-abelian varieties, the theory introduced in~\cite{EGVdL} applies to this situation. Hence the higher Hopf formulae obtained in~\cite{CKL} may be computed using higher-dimensional central extensions~\cite{EGVdL} instead of \v Cech derived functors. In principle, the only difficulty now lies in giving an explicit characterisation of the $m$-fold central extensions of Leibniz $n$-algebras. It turns out, however, that such an explicit characterisation is not hard to find at all. Moreover, this characterisation is easily adapted to work in the case of Lie $n$-algebras.

\subsection*{Structure of the text}
In the following section we recall the basic ideas behind semi-abelian homology with, in particular, the approach based on categorical Galois theory and higher-dimensional central extensions. Section~\ref{Section-Leibniz} treats homology of Leibniz $n$-algebras from this perspective. It contains a characterisation of $m$-fold central extensions of Leibniz $n$-algebras (Propositions~\ref{Proposition-Central-Extension} and~\ref{Proposition-m-Fold-Central-Extension}) and a Galois-theoretic proof of the higher Hopf formulae obtained in~\cite{CKL} (Theorem~\ref{Theorem-Leibniz-Vect}). Next, in Section~\ref{Section-Lie}, an analogous theory is built up for Lie $n$-algebras: we characterise the central extensions (Proposition~\ref{Proposition-m-Fold-Central-Extension-Lie}) and we obtain higher Hopf formulae (Theorem~\ref{Theorem-Lie-Vect}). Section~\ref{Section-Relative} is devoted to the homology theory which arises when the reflector from $\nLeib$ to $\nLie$ is derived. We characterise the central extensions of Leibniz $n$-algebras relative to the subvariety $\nLie$ (Propositions~\ref{Proposition-Relative-Central-Extension} and \ref{Proposition-m-Fold-Central-Extension-Relative}) and obtain a Hopf style formula for the relative homology (Theorem~\ref{Theorem-Leibniz-Lie}). The final Section~\ref{Section-UCE} contains some results on the universal central extensions induced by these three homology theories and on the relations between them: Propositions~\ref{UCE-Leibniz}, \ref{UCE-Lie}, \ref{UCE-Leibniz-Lie} and~\ref{Proposition-UCE-Leibniz-vs-Lie}.

\section{Preliminaries}
We sketch the basic ideas behind the theory of higher central extensions in semi-abelian categories.

\subsection{Semi-abelian categories}
We shall be using methods which were developed in the general framework of semi-abelian categories. Here it suffices to recall that a category is \defn{semi-abelian} when it is pointed, Barr exact and Bourn protomodular with binary coproducts~\cite{Janelidze-Marki-Tholen}. In this context there is a suitable notion of short exact sequence, and the basic homological lemmas such as the Snake Lemma and the $3\times3$~Lemma are valid~\cite{Borceux-Bourn}. Furthermore, homology of simplicial objects is well-behaved, so that Barr and Beck's definition of comonadic homology~\cite{Barr-Beck} may be extended as follows~\cite{EverVdL2}.

We write $\ker f\colon {K[f]\to B}$ for the kernel of a morphism $f\colon{B\to A}$.

\subsection{Comonadic homology}
Let $\A$ be a category and $\B$ a semi-abelian category. Let $I\colon{\A\to \B}$ be a functor and $\G$ a comonad on $\A$. If $A$ is an object of $\A$ and $m\geq 0$ then
\[
H_{m+1}(A,I)_{\G}=H_{m}NI\G A
\]
is the \defn{$(m+1)$-st homology object of $A$ with coefficients in $I$ relative to~$\G$}. Here $N\colon{\S\B\to \Ch\B}$ denotes the \defn{Moore functor}, which maps a simplicial object~$S$ to its normalised chain complex $NS$. It has objects
\[
N_{n} S=\bigcap_{i=0}^{n-1}K[\del_{i}\colon S_{n}\to S_{n-1}]
\]
for $n>0$, $N_{0}S=S_{0}$ and $N_{n}S=0$ for $n<0$, and boundary operators
\[
d_{n}=\del_{n}\comp \bigcap_{i}\ker \del_{i}\colon N_{n} S\to N_{n-1} S
\]
for $n\geq 1$. As shown in~\cite[Theorem~3.6]{EverVdL2}, the  boundary operators $d_{n}$ are always \defn{proper}---their images are kernels---so that computing the homology of $NS$ makes sense.

The simplicial object $\G A$ is part of the simplicial resolution of $A$ induced by the comonad
\[
\G =(G\colon \A \to \A ,\quad \delta\colon G\To G^{2} ,\quad \epsilon\colon G\To 1_{\A})
\]
on $\A$. Putting
\[
\del_{i}=G^{i}\epsilon_{G^{n-i}A}\colon G^{n+1}A\to G^nA, \qquad \sigma_{i}=G^{i}\delta_{G^{n-i}A}\colon G^{n+1}A\to G^{n+2}A ,
\]
for $0\leq i \leq n$, gives the sequence $(G^{n+1}A)_{n\geq 0}$ the structure of a simplicial object~$\G A$ of~$\A$. It has an augmentation $\epsilon_{A}\colon {GA\to A}$; this augmented simplicial object $(\G A,\epsilon_{A} \colon {GA\to A})$ is  the \defn{canonical $\G$-simplicial resolution} of~$A$.

If $\A$ is a semi-abelian variety and $\B$ is a subvariety of $\A$, one may consider the canonical comonad $\G$ on $\A$ induced by the forgetful/free adjunction to $\Set$, the category of sets, and one may take $I$ to be the left adjoint to the inclusion functor. For instance, every variety of $\Omega$-groups (i.e., every variety of universal algebras that has amongst its operations and identities those of the variety of groups and just one constant~\cite{Higgins}) is semi-abelian. Hence a choice of a subvariety here induces a canonical homology theory.

The main result of~\cite{EGVdL} gives an interpretation of such a homology theory  in terms of higher Hopf formulae. The technique used to obtain these Hopf formulae is based on categorical Galois theory, with in particular the theory of higher central extensions. We recall some of the basic concepts; see~\cite{EGVdL} for more details.

\subsection{Higher arrows, higher extensions, higher presentations}
For any integer $m\geq 1$, let $\langle m \rangle$ denote the set $\{1,\dots, m\}$; write $\langle 0\rangle =\varnothing$. The set $2^{\langle m \rangle}$ of all subsets of $\langle m \rangle$, ordered by inclusion, is considered as a category in the usual way: an inclusion $I\subseteq J$ in ${\langle m \rangle}$ corresponds to a map $!^{I}_{J}\colon {I\to J}$ in $2^{\langle m \rangle}$.

Let $\A$ be a semi-abelian category. The functor category $\Fun(2^{\langle m \rangle},\A)$ is denoted $\Arrm\A$. The objects of this category are called \defn{$m$-fold arrows} in~$\A$ (\defn{$m$-cubes} in~\cite{Brown-Ellis, CKL, Donadze-Inassaridze-Porter}). We write $f_I$ for $f(I)$ and $f^I_J$ for $f(!^I_J)$. When, in particular, $I$ is $\varnothing$ and~$J$ is a singleton $\{j \}$, we write $f_{j}=f^{\varnothing}_{\{j \}}$.

Any $(m+1)$-fold arrow $f$ may be considered as a morphism between $m$-fold arrows as follows: if $B$ denotes the restriction of $f$ to $\langle m \rangle$ and $A$ is its restriction to ${\{I\subset \langle m+1 \rangle\mid m+1\in I\}}$, then $f$ is a natural transformation from $B$ to $A$, i.e., a morphism $f\colon{B\to A}$ in $\Arrm\A$.

An $m$-fold arrow $f$ is called an \defn{extension} (\defn{exact $m$-cube} in~\cite{Brown-Ellis, CKL, Donadze-Inassaridze-Porter}) if, for all $I\subsetneq \langle m \rangle$, the induced morphism ${f_{I}\to\lim_{J\supset I} f_{J}}$ is a regular epimorphism. This notion of extension is easily seen to coincide with the one defined inductively in~\cite{EGVdL}. In particular, a one-fold extension (usually just called \defn{extension}) is a regular epimorphism in $\A$, and a two-fold extension (usually called a \defn{double extension}) is a pushout square. The full subcategory of $\Arrm\A$ determined by the $m$-fold extensions is denoted $\Extm\A$. When $m\geq 1$ this category is generally no longer semi-abelian, so concepts such as exact sequences, etc.\ will be considered in the semi-abelian category~$\Arrm\A$ instead. We further write $\Ext\A=\Ext^{1}\!\A$ and ${\Arr\A=\Arr^{1}\!\A}$.

Given an object $A$ of $\A$, an $m$-extension $f$ is said to be an \defn{$m$-fold presentation of $A$} (a \defn{free exact $m$-presentation of $A$} in~\cite{Brown-Ellis, CKL, Donadze-Inassaridze-Porter}) when $f_{\langle m \rangle}=A$ and $f_{I}$ is projective for all $I\subsetneq \langle m \rangle$. In the varietal case, canonical presentations may be obtained through truncations of canonical simplicial resolutions.

\subsection{Higher central extensions}
A \defn{Birkhoff subcategory}~\cite{Janelidze-Kelly} of a semi-abelian category is a full and reflective subcategory which is closed under subobjects and regular quotients. For instance, a Birkhoff subcategory of a variety of universal algebras is the same thing as a subvariety. Given a semi-abelian category $\A$ and a Birkhoff subcategory $\B$ of $\A$, we denote the induced adjunction
\begin{equation}\label{Birkhoff-Adjunction}
\xymatrix@1{{\A} \ar@<1ex>[r]^-{I} \ar@{}[r]|-{\perp} & {\B.} \ar@<1ex>[l]^-{\supset}}
\end{equation}
Together with the classes of extensions in $\A$ and $\B$, this adjunction forms a \emph{Galois structure} in the sense of Janelidze (\cite{Janelidze:Pure}; see also~\cite{Borceux-Janelidze}). The coverings with respect to this Galois structure are the \defn{central extensions} introduced in~\cite{Janelidze-Kelly}. (See Section~\ref{Dimension-one} below for an explicit definition.) These central extensions in turn determine a reflective subcategory $\CExt_{\B}\A$ of $\Ext\A$, which together with the appropriate classes of double extensions again forms a Galois structure. The coverings with respect to this Galois structure are called \defn{double central extensions}. This process may be repeated \emph{ad infinitum}, on each level inducing an adjunction
\begin{equation}\label{Birkhoff-Adjunction-High}
\xymatrix@1{{\Extm\A} \ar@<1ex>[r]^-{I_{m}} \ar@{}[r]|-{\perp} & {\CExtm\A} \ar@<1ex>[l]^-{\supset}}
\end{equation}
between the $m$-fold extensions in $\A$ and the \defn{$m$-fold central extensions} in~$\A$. In order to understand the higher Hopf formulae, we need an explicit description of those higher central extensions. We now sketch how, in general, the functors $I_{m}$ work.

\subsection{Dimension zero}
For every object $A$ of $\A$, the adjunction~\eqref{Birkhoff-Adjunction} induces a short exact sequence
\[
\xymatrix{0 \ar[r] & [A]_{\B} \ar[r]^-{\mu_{A}} & A \ar[r]^-{\eta_{A}} & IA \ar[r] & 0.}
\]
Here the object $[A]_{\B}$, defined as the kernel of $\eta_{A}$, acts as a zero-dimensional commutator relative to $\B$. Of course, $IA=A/[A]_{\B}$, so that $A$ is an object of~$\B$ if and only if $[A]_{\B}$ is zero. We shall also write $L_{0}[A]$ for this object~$[A]_{\B}$.

\subsection{Dimension one}\label{Dimension-one}
An extension $f\colon{B\to A}$ in $\A$ is central with respect to $\B$ or \defn{$\B$-central} if and only if the restrictions $[f_{0}]_{\B}$, $[f_{1}]_{\B}\colon{[R[f]]_{\B}\to [B]_{\B}}$ of the kernel pair projections $f_{0}$, $f_{1}\colon{R[f]\to B}$ coincide. This is the case precisely when $[f_{0}]_{\B}$ and $[f_{1}]_{\B}$ are isomorphisms, or, equivalently, when the kernel
\[
\ker[f_{0}]_{\B}\colon{L_{1}[f]\to [R[f]]_{\B}}
\]
of $[f_{0}]_{\B}$ is zero.
\[
\xymatrix{& L_{1}[f] \ar[d]_-{\ker [f_{0}]_{\B}} \ar@{.>}[ddr] \\
0 \ar[r] & [R[f]]_{\B} \ar[r]^-{\mu_{R[f]}} \ar@<-.5ex>[d]_-{[f_{0}]_{\B}} \ar@<.5ex>[d]^-{[f_{1}]_{\B}} & R[f] \ar@<-.5ex>[d]_-{f_{0}} \ar@<.5ex>[d]^-{f_{1}} \ar[r]^-{\eta_{R[f]}} & I R[f] \ar@<-.5ex>[d]_-{I f_{0}} \ar@<.5ex>[d]^-{I f_{1}} \ar[r] & 0\\
0 \ar[r] & [B]_{\B} \ar[r]_-{\mu_{B}} & B \ar[r]_-{\eta_{B}} & IB \ar[r] & 0}
\]
Through the composite $f_{1}\comp\mu_{R[f]}\comp\ker[f_{0}]_{\B}$ the object $L_{1}[f]$ may be considered as a normal subobject of $B$. It acts as a one-dimensional commutator relative to $\B$ and, if $K$ denotes the kernel of $f$, it is usually written $[K,B]_{\B}$. One computes the \defn{centralisation} $I_{1}f$ of $f$, i.e., its reflection into the subcategory $\CExt_{\B}\A$ of $\Ext\A$, by dividing out this commutator. This yields a morphism of short exact sequences
\[
\xymatrix{0 \ar[r] & L_{1}[f] \ar[r] \ar[d]_{[f]_{\CExt_{\B}\A}} & B \ar[r] \ar[d]_-{f} & \tfrac{B}{[K,B]_{\B}} \ar[r] \ar[d]^-{I_{1}f} & 0\\
& 0 \ar[r] & A \ar[r] & A \ar[r] & 0}
\]
in $\A$ which may be considered as a short exact sequence
\[
\xymatrix{0 \ar[r] & [f]_{\CExt_{\B}\A} \ar[r]^-{\mu^{1}_{f}} & f \ar[r]^-{\eta^{1}_{f}} & I_{1}f \ar[r] & 0}
\]
in $\Ext\A$. A crucial point here is that the extension $[f]_{\CExt_{\B}\A}$, which is to be divided out of the extension $f$ to obtain $I_{1}f$, is completely determined by an object $L_{1}[f]$ in~$\A$. This remains true in all higher dimensions.

\subsection{Higher dimensions}
For any $m\geq 1$, it may be shown that the object $[f]_{\CExtm\A}$ in the short exact sequence
\[
\xymatrix{0 \ar[r] & [f]_{\CExtm\A} \ar[r]^-{\mu^{m}_{f}} & f \ar[r]^-{\eta^{m}_{f}} & I_{m}f \ar[r] & 0}
\]
induced by the centralisation of an $m$-fold extension $f$ via the adjunction~\eqref{Birkhoff-Adjunction-High} is zero everywhere except in its ``top object'' $L_{m}[f]=([f]_{\CExtm\A})_{\varnothing}$. In parallel with the case $m=1$, this object $L_{m}[f]$ of $\A$ acts as an $m$-dimensional commutator which may be computed as the kernel of the restriction of $(f_{0})_{\varnothing}\colon{R[f]_{\varnothing}\to B_{\varnothing}}$ to a morphism  ${L_{m-1}[R[f]]\to L_{m-1}[B]}$. Likewise, an $m$-fold extension $f$ is central if and only if the induced morphisms
\[
(f_{0})_{\varnothing}, (f_{1})_{\varnothing}\colon{L_{m-1}[R[f]]\to L_{m-1}[B]}
\]
of $\A$ coincide.

\section{The case of Leibniz $n$-algebras}\label{Section-Leibniz}
In this section we recall the definition of Leibniz $n$-algebras and some of their basic properties. We characterise the $m$-fold central extensions of Leibniz $n$-algebras (Propositions~\ref{Proposition-Central-Extension} and~\ref{Proposition-m-Fold-Central-Extension}) and use this characterisation to give a Galois-theoretic proof of the higher Hopf formulae obtained in~\cite{CKL} (Theorem~\ref{Theorem-Leibniz-Vect}).

\subsection{The category $\nLeib$}
Throughout the text, we fix a natural number $n\geq 2$ and a ground field $\K$. Recall that a \defn{Leibniz $n$-algebra}~\cite{CLP} consists of a $\K$-vector space~$\LL$ with an additional $n$-ary operation $[-,\dots,-]$ that is $n$-linear and satisfies the fundamental identity
\begin{equation}\label{Leibniz-Identity}
[[l_{1},\dots,l_{n}],l'_{1},\dots,l'_{n-1}]=\sum_{1\leq i\leq n}[l_{1},\dots,l_{i-1},[l_{i},l'_{1},\dots,l'_{n-1}],l_{i+1},\dots,l_{n}],
\end{equation}
for all $l_{1}$, \dots, $l_{n}$, $l'_{1}$, \dots, $l'_{n-1}\in \LL$. By $n$-linearity, $[-,\dots,-]$ induces a linear map $\LL^{\tensor n}\to \LL$ usually called the \defn{bracket} of~$\LL$.

Since the underlying vector space $\LL$ is, in particular, an (abelian) group, the category $\nLeib$ of Leibniz $n$-algebras is a variety of $\Omega$-groups and as such, Leibniz $n$-algebras form a semi-abelian category with enough projectives. The notions of ($n$-sided) ideal, quotient, etc.\ considered in~\cite{Casas:hwtcolna} coincide with the categorical notions of normal subobject, cokernel, etc.

\subsection{The adjunction to vector spaces}
Any $\K$-vector space $\LL$ may be considered as a Leibniz $n$-algebra by equipping it with the trivial bracket: $[l_{1},\dots,l_{n}]=0$ for $l_{1}$, \dots, $l_{n}\in \LL$. The image of this inclusion ${\Vect\to \nLeib}$ consists precisely of the \defn{abelian} Leibniz $n$-algebras: those that admit an internal abelian group structure. Proving this fact is not difficult and analogous to the case of rings~\cite[Example 1.4.10]{Borceux-Bourn}. Hence the left adjoint $\abnleib\colon{\nLeib\to \Vect}$ to this inclusion is the abelianisation functor; it may be described as follows.

Given $n$ ideals $\NN_{1}$, \dots, $\NN_{n}$ of a Leibniz $n$-algebra $\LL$, we shall write $[\NN_{1},\dots,\NN_{n}]$ for the \defn{commutator ideal} of $\LL$ generated by the elements $[l_{1},\dots,l_{n}]$, where either $(l_{1},\dots, l_{n})$ or any of its permutations is in $\NN_{1}\times\dots\times \NN_{n}$. (In case $\NN_{2}=\dots=\NN_{n}=\LL$, the subspace of $\LL$ generated by those elements is automatically an ideal.) Clearly, a Leibniz $n$-algebra $\LL$ is abelian if and only if $[\LL,\dots,\LL]$ is zero; hence, for any $\LL$ in $\nLeib$, its reflection $\abnleib (\LL)$ into $\Vect$ is $\LL/[\LL,\dots,\LL]$.

\subsection{Central extensions}
The central extensions induced by this adjunction are the ones we expect them to be, i.e., the ones introduced in~\cite{Casas:oceolna}:

\begin{proposition}\label{Proposition-Central-Extension}
In $\nLeib$, an extension $f\colon{\BB\to \AA}$ with kernel $\KK$ is central if and only if the ideal $[\KK,\BB\dots,\BB]$ is zero. Hence, in any case, the object $L_{1}[f]$ is equal to this ideal of $\BB$.
\end{proposition}
\begin{proof}
By definition, the extension $f$ is central if and only if the restrictions of the kernel pair projections $f_{0}$, $f_{1}$ to the brackets
\[
[R[f],\dots,R[f]]\to [\BB,\dots,\BB]
\]
coincide.

This latter condition implies
\begin{align*}
[b_{1},\dots,b_{n}]&=f_{0}[(b_{1},b_{1}),\dots, (b_{i-1},b_{i-1}), (b_{i},0), (b_{i+1},b_{i+1}),\dots,(b_{n},b_{n})]\\
&=f_{1}[(b_{1},b_{1}),\dots, (b_{i-1},b_{i-1}), (b_{i},0), (b_{i+1},b_{i+1}),\dots,(b_{n},b_{n})]\\
&=[b_{1},\dots,b_{i-1},0,b_{i+1},\dots,b_{n}]=0
\end{align*}
for any $b_{1}$, \dots, $b_{n}\in \BB$ with $b_{i}\in \KK$. Hence when $f$ is central $[\KK,\BB,\dots,\BB]$ is zero.

Now let $[(b_{1},b_{1}+k_{1}),\dots, (b_{n},b_{n}+k_{n})]$ be a generator of $[R[f],\dots,R[f]]$; here $b_{1}$, \dots, $b_{n}\in \BB$ and $k_{1}$, \dots, $k_{n}\in \KK$. Then
\[
f_{0}[(b_{1},b_{1}+k_{1}),\dots, (b_{n},b_{n}+k_{n})]=[b_{1},\dots,b_{n}],
\]
while also
\[
f_{1}[(b_{1},b_{1}+k_{1}),\dots, (b_{n},b_{n}+k_{n})]=[b_{1}+k_{1},\dots, b_{n}+k_{n}]=[b_{1},\dots,b_{n}]
\]
since the bracket is $n$-linear and $[\KK,\BB,\dots,\BB]$ is zero. This implies that $f$ is a central extension.
\end{proof}

\subsection{Higher central extensions}
Using essentially the same proof as in~\ref{Proposition-Central-Extension}, this result may be extended to higher dimensions as follows.

\begin{proposition}\label{Proposition-m-Fold-Central-Extension}
Consider $m\geq 1$. An $m$-fold extension $f\colon{\BB\to \AA}$ in $\nLeib$ is central if and only if the object
\begin{equation}\label{Object-L}
\sum_{I_{1}\cup\dots\cup I_{n}=\langle m \rangle}\Bigl[\bigcap_{i\in I_{1}}K[f_{i}],\dots, \bigcap_{i\in I_{n}}K[f_{i}]\Bigr]
\end{equation}
is zero. Hence, in any case, it is equal to $L_{m}[f]$.
\end{proposition}
\begin{proof}
We give a proof by induction on $m$. Since the case $m=1$ was considered in Proposition~\ref{Proposition-Central-Extension}, we may suppose that the result holds for $m-1$. Hence the $m$-extension $f$ is central if and only if the restrictions of the kernel pair projections $f_{0}$, $f_{1}$ to morphisms
\[
\sum_{\stackrel{I_{1}\cup\dots\cup I_{n}}{=\langle m-1 \rangle}}\Bigl[\bigcap_{i\in I_{1}}K[R[f]_{i}],\dots, \bigcap_{i\in I_{n}}K[R[f]_{i}]\Bigr]\to \sum_{\stackrel{I_{1}\cup\dots\cup I_{n}}{=\langle m-1 \rangle}}\Bigl[\bigcap_{i\in I_{1}}K[\BB_{i}],\dots, \bigcap_{i\in I_{n}}K[\BB_{i}]\Bigr]
\]
coincide. 

In what follows, the arguments given for generators may easily be extended to general elements.

Suppose that the latter condition holds, let $I_{1}\cup\dots\cup I_{n}$ be a partition of $\langle m \rangle$ and consider $[k_{1},\dots, k_{n}]$ in
\[
\Bigl[\bigcap_{i\in I_{1}}K[f_{i}],\dots, \bigcap_{i\in I_{n}}K[f_{i}]\Bigr]\subset f_{\varnothing}.
\]
Suppose that $m\in I_{j}$. Then $k_{j}\in K[f_{m}]$, so that
\[
[(k_{1},k_{1}),\dots, (k_{j-1},k_{j-1}),(k_{j},0),(k_{j+1},k_{j+1}),\dots,(k_{n},k_{n})]
\]
is in $R[f]_{\varnothing}$; in fact, it is easily seen to be an element of
\[
\Bigl[\bigcap_{i\in I_{1}}K[R[f]_{i}],\dots,\bigcap_{i\in I_{j}\setminus \{m\}}K[R[f]_{i}],\dots, \bigcap_{i\in I_{n}}K[R[f]_{i}]\Bigr].
\]
Hence
\begin{align*}
[k_{1},\dots, k_{n}]&=(f_{0})_{\varnothing}[(k_{1},k_{1}),\dots, (k_{j-1},k_{j-1}),(k_{j},0),(k_{j+1},k_{j+1}),\dots,(k_{n},k_{n})]\\
&=(f_{1})_{\varnothing}[(k_{1},k_{1}),\dots, (k_{j-1},k_{j-1}),(k_{j},0),(k_{j+1},k_{j+1}),\dots,(k_{n},k_{n})]\\
&=[k_{1},\dots,k_{j-1},0,k_{j},\dots, k_{n}]=0,
\end{align*}
which proves that~\eqref{Object-L} is zero.

Conversely, suppose that~\eqref{Object-L} is zero, let $I_{1}\cup\dots\cup I_{n}$ be a partition of $\langle m-1 \rangle$ and consider $[(b_{1},b_{1}+k_{1}),\dots, (b_{n},b_{n}+k_{n})]$ in
\[
\Bigl[\bigcap_{i\in I_{1}}K[R[f]_{i}],\dots, \bigcap_{i\in I_{n}}K[R[f]_{i}]\Bigr];
\]
here $b_{1}$, \dots, $b_{n}\in \BB_{\varnothing}$ and $k_{1}$, \dots, $k_{n}\in K[f_{m}]$. Now
\[
(f_{0})_{\varnothing}[(b_{1},b_{1}+k_{1}),\dots, (b_{n},b_{n}+k_{n})]=[b_{1},\dots,b_{n}],
\]
while also
\begin{align*}
&(f_{1})_{\varnothing}[(b_{1},b_{1}+k_{1}),\dots, (b_{n},b_{n}+k_{n})]=[b_{1}+k_{1},\dots, b_{n}+k_{n}]\\
&=[b_{1},b_{2}+k_{2},\dots, b_{n}+k_{n}]+[k_{1},b_{2}+k_{2},\dots, b_{n}+k_{n}]\\
&=[b_{1},b_{2},b_{3}+k_{3},\dots, b_{n}+k_{n}]+[b_{1},k_{2},b_{3}+k_{3},\dots, b_{n}+k_{n}]\\
&=\cdots=[b_{1},\dots,b_{n}]
\end{align*}
since~\eqref{Object-L} is zero.
\end{proof}

\subsection{Homology}
In his article~\cite{Casas:hwtcolna}, Casas studied a homology theory for Leibniz $n$-algebras based on Leibniz homology~\cite{Loday-Leibniz, LP}. The latter homology theory is extended to Leibniz $n$-algebras via the \defn{Daletskii functor}~\cite{Daletskii}
\[
\D_{n-1}\colon{\nLeib\to {}_{2}\Leib=\Leib}
\]
which takes a Leibniz $n$-algebra $\LL$ and maps it to the Leibniz algebra $\D_{n-1}(\LL)$ with underlying vector space $\LL^{\tensor (n-1)}$ and bracket
\[
[l_{1}\tensor\cdots\tensor l_{n-1},l'_{1}\tensor\cdots\tensor l'_{n-1}]=\sum_{1\leq i\leq n-1}l_{1}\tensor\cdots\tensor[l_{i},l'_{1},\dots, l'_{n-1}]\tensor\cdots\tensor l_{n-1}.
\]
It is explained in~\cite{Casas:hwtcolna} that the underlying vector space of a Leibniz $n$-algebra $\LL$ always carries a structure of $\D_{n-1}(\LL)$-corepresentation. By definition, the $m$-th homology of~$\LL$ is
\[
\nHL_{m}(\LL)=\HL_{m}(\D_{n-1}(\LL),\LL),
\]
i.e., the homology of the associated Leibniz algebra $\D_{n-1}(\LL)$ with coefficients in the $\D_{n-1}(\LL)$-corepresentation $\LL$.

Among other results, in~\cite{Casas:hwtcolna} a Hopf formula expressing the first homology of a Leibniz $n$-algebra as a quotient of commutators is obtained. (It is the \emph{first} homology rather than the \emph{second} homology due to a dimension shift caused by this particular definition.) Later Casas, Khmaladze and Ladra also characterised the higher homology vector spaces in terms of Hopf formulae~\cite{CKL}.

In the latter article it is also shown that this concept of homology for Leibniz $n$-algebras coincides with the Quillen homology \cite[Theorem~4]{CKL}. On the other hand, the category $\nLeib$, being a semi-abelian variety, admits a canonical homology theory, namely the comonadic homology with coefficients in the abelianisation functor, relative to the canonical comonad $\G$ on $\nLeib$. Since both may be expressed as Quillen homology, up to a dimension shift the two homology theories coincide:
\[
H_{m+1}(\LL,\abnleib)_{\G}\cong\nHL_{m}(\LL),
\]
for all $\LL$ in $\nLeib$ and $m\geq 0$. As such, the standard techniques of semi-abelian homology are available, and Proposition~\ref{Proposition-m-Fold-Central-Extension} provides us with an alternative proof for the next theorem.

\begin{theorem}[Hopf type formula, Theorem 17 in \cite{CKL}]\label{Theorem-Leibniz-Vect}
Consider $m\geq 1$. If $f$ is an $m$-fold presentation of a Leibniz $n$-algebra $\LL$, then
\[
H_{m+1}(\LL,\abnleib)_{\G}\cong\frac{[f_{\varnothing},\dots,f_{\varnothing}]\cap \bigcap_{i\in \langle m\rangle}K[f_{i}]}{\sum_{I_{1}\cup\dots\cup I_{n}=\langle m \rangle}\Bigl[\bigcap_{i\in I_{1}}K[f_{i}],\dots, \bigcap_{i\in I_{n}}K[f_{i}]\Bigr]}.
\]
\end{theorem}
\begin{proof}
This is a combination of Proposition~\ref{Proposition-m-Fold-Central-Extension} with~\cite[Theorem~8.1]{EGVdL}.
\end{proof}

\section{The case of Lie $n$-algebras}~\label{Section-Lie}
The theory for Lie $n$-algebras is almost literally the same as in the Leibniz $n$-algebra case. Before stating the main results, let us first recall some basic definitions. In order to recover the case of Lie algebras for $n=2$, we assume henceforward that the characteristic of the ground field $\K$ is not equal to $2$.

A \defn{Lie $n$-algebra}~\cite{Filippov} is a skew-symmetric Leibniz $n$-algebra, i.e., it is a vector space $\LL$ equipped with a linear map $[-,\dots,-]\colon{\LL^{\tensor n}\to \LL}$ that satisfies the identity~\eqref{Leibniz-Identity} as well as the further identity
\[
[l_{1},\dots,l_{n}]=\sgn(\sigma)[l_{\sigma(1)},\dots,l_{\sigma(n)}]
\]
called \defn{skew symmetry}. Here $l_{1}$, \dots, $l_{n}$ are elements of $\LL$, $\sigma \in S_{n}$ is a permutation of $\{1,\dots,n\}$ and $\sgn(\sigma)\in\{-1,1\}$ is the signature of $\sigma$. The category $\nLie$ of Lie $n$-algebras being a subvariety of $\nLeib$, it is a semi-abelian variety. The canonical comonad on $\nLie$ is denoted by $\Hbb$. The abelianisation functor for Leibniz $n$-algebras restricts to a functor $\abnlie\colon{\nLie \to \Vect}$ which is left adjoint to the inclusion functor.

Given $n$ ideals $\NN_{1}$, \dots, $\NN_{n}$ of a Lie $n$-algebra $\LL$, by skew symmetry the object $[\NN_{1},\dots,\NN_{n}]$ is the ideal of $\LL$ generated by the brackets $[l_{1},\dots,l_{n}]$ for all $l_{1}\in \NN_{1}$, \dots, $l_{n}\in \NN_{n}$. (Here no further permutations of $(l_{1},\dots,l_{n})$ are necessary.) Again, a Lie $n$-algebra $\LL$ is abelian if and only if $[\LL,\dots,\LL]$ is zero; hence, for any $\LL$ in $\nLie$, its reflection $\abnlie (\LL)$ into $\Vect$ is $\LL/[\LL,\dots,\LL]$.

\begin{proposition}\label{Proposition-m-Fold-Central-Extension-Lie}
Consider $m\geq 1$. An $m$-fold extension $f\colon{\BB\to \AA}$ in $\nLie$ is central if and only if the object
\[
\sum_{I_{1}\cup\dots\cup I_{n}=\langle m \rangle}\Bigl[\bigcap_{i\in I_{1}}K[f_{i}],\dots, \bigcap_{i\in I_{n}}K[f_{i}]\Bigr]
\]
is zero.
\end{proposition}
\begin{proof}
The proof of Proposition~\ref{Proposition-m-Fold-Central-Extension} may be copied.
\end{proof}

\begin{theorem}[Hopf type formula for homology of Lie $n$-algebras]\label{Theorem-Lie-Vect}
Consider $m\geq 1$. If $f$ is an $m$-fold presentation of a Lie $n$-algebra $\LL$, then
\[
H_{m+1}(\LL,\abnlie)_{\Hbb}\cong\frac{[f_{\varnothing},\dots,f_{\varnothing}]\cap \bigcap_{i\in \langle m\rangle}K[f_{i}]}{\sum_{I_{1}\cup\dots\cup I_{n}=\langle m \rangle}\Bigl[\bigcap_{i\in I_{1}}K[f_{i}],\dots, \bigcap_{i\in I_{n}}K[f_{i}]\Bigr]}.
\]
\end{theorem}
\begin{proof}
This is a combination of Proposition~\ref{Proposition-m-Fold-Central-Extension-Lie} with~\cite[Theorem~8.1]{EGVdL}.
\end{proof}

\section{The relative case}\label{Section-Relative}
Now we consider the homology theory which arises when the reflector from $\nLeib$ to $\nLie$ is derived. We characterise the central extensions of Leibniz $n$-algebras relative to the subvariety $\nLie$ (Propositions~\ref{Proposition-Relative-Central-Extension} and \ref{Proposition-m-Fold-Central-Extension-Relative}) and obtain a Hopf style formula for the relative homology (Theorem~\ref{Theorem-Leibniz-Lie}).

\subsection{The Liesation of a Leibniz $n$-algebra}
The inclusion of $\nLie$ into $\nLeib$ has a left adjoint, called the \defn{Liesation functor} and denoted
\[
\nlie\colon{\nLeib\to \nLie}.
\]
Let $\LL$ be a Leibniz $n$-algebra. Given a permutation $\sigma\in S_{n}$ and elements $l_{1}$, \dots, $l_{n}\in \LL$, we write
\[
\langle l_{1},\dots, l_{n}\rangle_{\sigma}=[l_{1},\dots,l_{n}]-\sgn(\sigma)[l_{\sigma(1)},\dots,l_{\sigma(n)}].
\]
Note that for any given $\sigma$, $\langle l_{1},\dots, l_{n}\rangle_{\sigma}$ is $n$-linear in the variables $l_{1}$, \dots, $l_{n}$ and thus determines a linear map $\langle -,\dots, -\rangle_{\sigma}\colon{\LL^{\tensor n}\to \LL}$. Given $n$ ideals $\NN_{1}$, \dots, $\NN_{n}$ of $\LL$, we write $\langle \NN_{1},\dots, \NN_{n}\rangle $ for the ideal of $\LL$ generated by all elements $\langle l_{1},\dots, l_{n}\rangle_{\sigma}$
where $l_{1}\in \NN_{1}$, \dots, $l_{n}\in \NN_{n}$ and $\sigma\in S_{n}$. We call this object the \defn{relative commutator} of $\NN_{1}$, \dots, $\NN_{n}$. Since a Leibniz $n$-algebra $\LL$ is a Lie $n$-algebra if and only if the relative commutator $\langle \LL,\dots, \LL\rangle$ is zero, the functor $\nlie$ maps the algebra $\LL$ to the quotient $\LL/\langle \LL,\dots, \LL\rangle$. Thus we obtain a commutative triangle of left adjoint functors:
\[
\xymatrix@!0@=4.5em{{\nLeib} \ar[rr]^-{\nlie} \ar[dr]_-{\abnleib} && \nLie \ar[dl]^-{\abnlie}\\
& \Vect}
\]
In case $n=2$ we regain the triangle of adjunctions considered in~\cite{CVdL}.

\begin{remark}
Since any permutation may be expressed as a composite of transpositions, the relative commutator $\langle \LL,\dots, \LL\rangle$ is generated by those brackets $[l_{1},\dots,l_{n}]$ in $\LL$ where $l_{i}=l_{i+1}$ for some $1\leq i\leq n-1$.
\end{remark}

\subsection{Relative central extensions}
We are now ready to characterise the central extensions of Leibniz $n$-algebras relative to the subvariety of Lie $n$-algebras. This is an extension of the case $n=2$ considered in~\cite{CVdL}.

\begin{proposition}\label{Proposition-Relative-Central-Extension}
An extension of Leibniz $n$-algebras $f\colon{\BB\to \AA}$ with kernel $\KK$ is central with respect to $\nLie$ if and only if the ideal $\langle \KK,\BB,\dots, \BB\rangle$ is zero.
\end{proposition}
\begin{proof}
The extension $f$ is central if and only if the restrictions of the kernel pair projections $f_{0},f_{1}\colon{R[f]\to \BB}$ to morphisms
\[
\langle R[f],\dots,R[f]\rangle \to \langle \BB,\dots \BB\rangle
\]
coincide.

If $x_{1}\in \KK$ and $x_{2}$, \dots, $x_{n}\in \BB$ then $y_{1}=(x_{1},0)$, $y_{2}=(x_{2},x_{2})$, \dots, $y_{n}=(x_{n},x_{n})$ are all elements of $R[f]$. Moreover, for all $\sigma \in S_{n}$,
\[
\langle x_{1},\dots, x_{n}\rangle_{\sigma}=f_{0}\langle y_{1},\dots, y_{n}\rangle_{\sigma},
\]
while also
\[
f_{1}\langle y_{1},\dots, y_{n}\rangle_{\sigma}=\langle 0,x_{2},\dots, x_{n}\rangle_{\sigma}=0.
\]
Hence $f$ being central implies that $\langle \KK,\BB,\dots, \BB\rangle=0$.

Conversely, any generator of $\langle R[f],\dots,R[f]\rangle$ may be written as
\[
\langle(b_{1},b_{1}+k_{1}),\dots, (b_{n},b_{n}+k_{n})\rangle_{\sigma}
\]
for some $b_{1}$, \dots, $b_{n}\in \BB$, $k_{1}$, \dots, $k_{n}\in \KK$ and $\sigma\in S_{n}$. Now if the ideal $\langle \KK,\BB,\dots, \BB\rangle$ is zero, $\langle b_{1},\dots b_{n}\rangle_{\sigma}$, the image of this generator through $f_{0}$, is equal to its image through $f_{1}$, since by $n$-linearity of $\langle -,\dots,-\rangle_{\sigma}$,
\begin{align*}
&\langle b_{1}+k_{1}, b_{2}+k_{2},\dots,b_{n}+k_{n}\rangle_{\sigma}\\
&=\langle b_{1},b_{2}+k_{2},\dots,b_{n}+k_{n}\rangle_{\sigma}+\langle k_{1},b_{2}+k_{2},\dots,b_{n}+k_{n}\rangle_{\sigma}\\
&=\langle b_{1},b_{2}+k_{2},\dots,b_{n}+k_{n}\rangle_{\sigma}=\dots=\langle b_{1},\dots b_{n}\rangle_{\sigma}.
\end{align*}
This proves that the extension $f$ is central.
\end{proof}

In case $n=2$, $\langle \KK,\BB\rangle=0$ if and only if $\KK \subseteq Z_{\Lie}(\BB)$, so we recover Proposition~4.3 in~\cite{CVdL} for central extensions relative to $\Lie$.

\subsection{Higher relative central extensions}
Using the ideas from the proof of Proposition~\ref{Proposition-Relative-Central-Extension}, it is now easy to adapt the proof of Proposition~\ref{Proposition-m-Fold-Central-Extension} to the relative case so that the following characterisation of $m$-fold $\nLie$-central extensions is obtained.

\begin{proposition}\label{Proposition-m-Fold-Central-Extension-Relative}
Consider $m\geq 1$. An $m$-fold extension $f\colon{\BB\to \AA}$ in $\nLeib$ is central with respect to $\nLie$ if and only if the object
\[
\sum_{I_{1}\cup\dots\cup I_{n}=\langle m \rangle}\Bigl\langle\bigcap_{i\in I_{1}}K[f_{i}],\dots, \bigcap_{i\in I_{n}}K[f_{i}]\Bigr\rangle
\]
is zero.\noproof
\end{proposition}

This now implies

\begin{theorem}[Hopf type formula for Leibniz $n$-algebras vs.\ Lie $n$-algebras]\label{Theorem-Leibniz-Lie}
Consider $m\geq 1$. If $f$ is an $m$-fold presentation of a Leibniz $n$-algebra $\LL$, then an isomorphism
\[
H_{m+1}(\LL,\nlie)_{\G}\cong\frac{\langle f_{\varnothing},\dots,f_{\varnothing}\rangle\cap \bigcap_{i\in \langle m\rangle}K[f_{i}]}{\sum_{I_{1}\cup\dots\cup I_{n}=\langle m \rangle}\Bigl\langle\bigcap_{i\in I_{1}}K[f_{i}],\dots, \bigcap_{i\in I_{n}}K[f_{i}]\Bigr\rangle}
\]
exists.\noproof
\end{theorem}

In case $n=2$ and $m=1$ this formula takes the usual shape
\[
H_{2}(\LL,\lie)_{\G}\cong\frac{\langle\FF,\FF\rangle\cap\RR}{\langle\RR,\FF\rangle}
\]
for
\[
\xymatrix{0 \ar[r] & \RR \ar[r] & \FF \ar[r] & \LL \ar[r] & 0}
\]
any free presentation of the Leibniz algebra $\LL$.

\section{Universal central extensions}\label{Section-UCE}
Let $\A$ be a semi-abelian category and $\B$ a Birkhoff subcategory of $\A$. Let~$I$ denote the reflector from $\A$ to $\B$. An object $A$ of $\A$ is called \defn{perfect with respect to $\B$} when $IA$ is zero. A $\B$-central extension $u\colon{U\to A}$ is called \defn{universal} when for every $\B$-central extension $f\colon{B\to A}$ there exists a unique map $\overline{f}\colon{U\to B}$ such that $f\comp \overline{f}=u$. An object $A$ admits a universal $\B$-central extension if and only if it is $\B$-perfect, in which case this universal $\B$-central extension $u^{I}_{A}\colon{U(A,I)\to A}$ is constructed as follows~\cite{CVdL}: given a $1$-presentation $f\colon{B\to A}$ with kernel $K$, the object $U(A,I)$ is the quotient $[B,B]_{\B}/[K,B]_{\B}$, and $u$ is the induced morphism.

The results from~\cite{CVdL} on universal central extensions in semi-abelian categories now particularise to the following generalisation of~\cite[Section~4.2]{CVdL} where the case $n=2$ is treated.

\begin{proposition}\label{UCE-Leibniz}
A $\Vect$-central extension of Leibniz $n$-algebras $u\colon{\UU\to\LL}$ is universal if and only if $H_{2}(\UU,\abnleib)_{\G}=H_{1}(\UU,\abnleib)_{\G}=0$.

A Leibniz $n$-algebra $\LL$ is perfect with respect to $\Vect$ if and only if $\LL=[\LL,\dots,\LL]$ if and only if $\LL$ admits a universal $\Vect$-central extension
\[
u_{\LL}^{\abnleib}\colon{U(\LL,\abnleib)\to\LL}
\]
with kernel $H_{2}(\LL,\abnleib)_{\G}$.
\end{proposition}
\begin{proof}
This is a combination of~\cite[Theorem~2.13]{CVdL} with \cite[Corollary~2.15]{CVdL}.
\end{proof}

The following explicit construction of the universal $\Vect$-central extension of a $\Vect$-perfect Leibniz $n$-algebra was obtained in~\cite{Casas:natp}.

\begin{construction}\label{Construction-UCE-nLb}
Let $\LL$ be a Leibniz $n$-algebra. We write $\UU$ for the quotient of the vector space $\LL^{\otimes n }$ by the subspace spanned by the elements of the form
\[
[l_1,\dots,l_n] \otimes {l'_2 \otimes \dots \otimes l'_{n}}
-\sum_{1\leq i\leq n}l_1 \otimes  \dots \otimes [l_i,l'_2,\dots,l'_n]
\otimes \dots \otimes l_n.
\]
The vector space $\UU$ inherits a structure of Leibniz $n$-algebra by means of the $n$-ary bracket
\[
 [l_1^1 * \dots *l_n^1, \dots, l_1^n *\dots *l_n^n] = [l_1^1, \dots,l_n^1] * \dots * [l_1^n, \dots,l_n^n],
\]
where $l_1 * \dots * l_n\in \UU$ denotes the equivalence class of $l_1 \otimes\dots \otimes l_n\in \LL^{\otimes n}$. If $\LL$ is a $\Vect$-perfect Leibniz $n$-algebra then by~\cite[Theorem~5]{Casas:natp} we have that the morphism of Leibniz $n$-algebras
\[
u\colon \UU\to {\LL}\colon l_1 * \dots *l_n \mapsto [l_1,\dots,l_n]
\]
is a universal $\Vect$-central extension of $\LL$.
\end{construction}

The corresponding result for Lie $n$-algebras looks as follows.

\begin{proposition}\label{UCE-Lie}
A $\Vect$-central extension of Lie $n$-algebras $u\colon{\UU\to\LL}$ is universal if and only if $H_{2}(\UU,\abnlie)_{\Hbb}=H_{1}(\UU,\abnlie)_{\Hbb}=0$.

A Lie $n$-algebra $\LL$ is perfect with respect to $\Vect$ if and only if $\LL=[\LL,\dots,\LL]$ if and only if it admits a universal $\Vect$-central extension
\[
u_{\LL}^{\abnlie}\colon{U(\LL,\abnlie)\to\LL}
\]
with kernel $H_{2}(\LL,\abnlie)_{\Hbb}$.\noproof
\end{proposition}

Next to the categorical construction recalled above, there is the following explicit construction---which does not use free objects---of the universal $\Vect$-central extension of a perfect Lie $n$-algebra.

\begin{construction}\label{Construction-UCE-nLie}
Let $\LL$ be a Lie $n$-algebra. We write $\UU$ for the quotient of the vector space $\LL^{\tensor n}$ by the subspace spanned by the elements of the form
\[
[l_1,l_2,\dots,l_n] \otimes {l'_2 \otimes \dots \otimes l'_{n}}
-\sum_{1\leq i\leq n}l_1 \otimes {l_2\otimes \dots \otimes [l_i,l'_2,\dots,l'_n]
\otimes \dots \otimes l_n}
\]
and all
\[
\text{$l_1 \otimes l_2 \otimes \dots \otimes l_n$ with $l_i=l_{i+1}$ for some $i\in\{1,\dots,n-1\}$.}
\]
Let $l_1\odot l_2\odot \dots\odot l_n$ denote the equivalence class of $l_1 \otimes l_2\otimes \dots \otimes l_n$. It is routine to check that $\UU$ carries a Lie $n$-algebra structure given by the $n$-ary bracket
\begin{equation}\label{construction}
 [l_1^1\odot l_2^1\odot \dots\odot l_n^1, \dots, l_1^n\odot l_2^n\odot
\dots \odot l_n^n] = [l_1^1,l_2^1, \dots,l_n^1]\odot \dots \odot [l_1^n,l_2^n,
\dots,l_n^n].
\end{equation}
Moreover, there is a well-defined morphism of Lie $n$-algebras
\[
u\colon {\UU\to\LL}\colon {l_1\odot l_2\odot \dots \odot l_n \mapsto [l_1,l_2,\dots,l_n]}.
\]
Now suppose that $\LL$ is a $\Vect$-perfect Lie $n$-algebra. Then it follows from the equality~\eqref{construction} that $\UU$ is also a $\Vect$-perfect Lie $n$-algebra. Furthermore, $u$ is an epimorphism of Lie $n$-algebras.

We claim that $u\colon {\UU\to {\LL}}$ is a universal $\Vect$-central extension of the $\Vect$-perfect Lie $n$-algebra $\LL$. Indeed, thanks again to the equality~\eqref{construction}, the ideal $[K[u],\UU,\dots,\UU]$ is trivial and thus $u$ is a $\Vect$-central extension. Given any other $\Vect$-central extension of $\LL$, say $f\colon{\HH \to \LL}$, there is a morphism of Lie $n$-algebras $\overline{f}\colon{\UU\to \HH}$ given by $\overline{f} {(l_1\odot l_2\odot \dots \odot l_n)}= [h_1,h_2,\dots,h_n]$, where $h_i\in \HH$ such that $f(h_i)=l_i$ for $1\leq i\leq n$. Here $\overline{f}$ is well-defined because of the centrality of ${f}\colon{\HH\to \LL}$. Clearly $f\comp\overline{f}=u$, and if $g\colon{\UU\to \HH}$ is another morphism with this property, then for any $x\in \UU$ there exists a $z$ in the centre of $\HH$ such that $g(x)=\overline{f}(x)+z$. It follows that $g[x_1,\dots,x_n]=\overline{f}[x_1,\dots,x_n]$, for all $x_i\in \UU$ and $1\leq i\leq n$. Since $\UU$ is a $\Vect$-perfect Lie $n$-algebra we deduce that $g=\overline{f}$.
\end{construction}

\begin{remark}
For $n=2$ the Lie algebra $\UU$ in Construction~\ref{Construction-UCE-nLie} is isomorphic to the non-abelian Lie exterior square of $\LL$. Thus if $\LL$ is a $\Vect$-perfect Lie algebra then it is the same as the non-abelian Lie tensor square of $\LL$~\cite{El0}. This fact follows easily from~\cite[Proposition 1]{IKL}. Moreover, we recover the description of the universal central extension of a $\Vect$-perfect Lie algebra obtained in~\cite[Theorem~11]{El0}.
\end{remark}

In the relative case we obtain the next result.

\begin{proposition}\label{UCE-Leibniz-Lie}
An $\nLie$-central extension of Leibniz $n$-algebras $u\colon{\UU\to\LL}$ is universal if and only if $H_{2}(\UU,\nlie)_{\G}=H_{1}(\UU,\nlie)_{\G}=0$.

A Leibniz $n$-algebra $\LL$ is perfect with respect to $\nLie$ if and only if $\LL=\langle\LL,\dots,\LL\rangle$ if and only if it admits a universal $\nLie$-central extension
\[
u_{\LL}^{\nlie}\colon{U(\LL,\nlie)\to\LL}
\]
with kernel $H_{2}(\LL,\nlie)_{\G}$.\noproof
\end{proposition}

Finally we describe the relation between the universal $\Vect$-central extensions in $\nLeib$ and $\nLie$. Given a $\Vect$-perfect Lie $n$-algebra $\LL$, it is also $\Vect$-perfect as a Leibniz $n$-algebra. The universal $\Vect$-central extension in
$\nLie$,
\[
u^{\abnlie}_{\LL}\colon U(\LL,\abnlie)\to {\LL},
\]
is a  $\Vect$-central extension in $\nLeib$ of $\LL$. Hence there is a morphism of Leibniz $n$-algebras $f\colon U(\LL,\abnleib)\to U(\LL,\abnlie)$ such that
$u^{\abnleib}=u^{\abnlie}\comp f$. Using the notations as in Construction \ref{Construction-UCE-nLb} and Construction \ref{Construction-UCE-nLie}, $f$ is explicitly given by
\[
f(l_1 * l_2 *\dots * l_n)=l_1\odot l_2\odot \dots\odot l_n,
\]
and in fact it is an epimorphism of Leibniz $n$-algebras. Restriction of $f$ to the kernel of $u^{\abnleib}$ yields an epimorphism from $H_{2}(\LL,\abnleib)_{\G}$  to $H_{2}(\LL,\abnlie)_{\Hbb}$ and we obtain the following proposition, which is also a special case of~\cite[Proposition~3.3]{CVdL}.

\begin{proposition}\label{Proposition-UCE-Leibniz-vs-Lie}
When a Lie $n$-algebra $\LL$ is perfect with respect to $\Vect$, we have a short exact sequence
\[
\xymatrix@1{0 \ar[r] & H_{2}(U(\LL,\abnlie),\abnleib)_{\G} \ar[r] & H_{2}(\LL,\abnleib)_{\G} \ar[r] & H_{2}(\LL,\abnlie)_{\Hbb} \ar[r] & 0.}
\]
Moreover,
\[
\langle H_{2}(\LL,\abnleib)_{\G},\dots,H_{2}(\LL,\abnleib)_{\G}\rangle = H_{2}(U(\LL,\abnlie),\abnleib)_{\G},
\]
and $u_{\LL}^{\abnleib}=u_{\LL}^{\abnlie}$ if and only if $H_{2}(\LL,\abnleib)_{\G} \cong H_{2}(\LL,\abnlie)_{\Hbb}$.\noproof
\end{proposition}

In case $n=2$ we recover Proposition~4.4 in~\cite{CVdL}.

\section*{Acknowledgements}
We would like to thank Tomas Everaert for some comments and suggestions, and the fourth author wishes to thank the University of Vigo for its kind hospitality during his stay in Pontevedra.



\end{document}